\documentclass[12pt]{amsart}
\numberwithin{equation}{section}
\usepackage{amsmath,amsthm,amsfonts,amscd,eucal,mathabx}
\usepackage{MnSymbol}
\usepackage{xypic}
\usepackage{graphicx}
\def\cb{{\mathcal B}}

\def\ch{{\mathcal H}}

\def\cs{{\mathcal S}}

\def\cx{{\mathcal X}}

\def\ga{{\mathfrak A}}
\def\gb{{\mathfrak B}}

\def\bc{{\mathbb C}}
\def\bd{{\mathbb D}}

\def\bm{{\mathbb M}}
\def\bn{{\mathbb N}}

\def\br{{\mathbb R}}
\def\bt{{\mathbb T}}

\def\a{\alpha}
\def\b{\beta}
\def\g{\gamma}  
\def\d{\delta}  

\def\eps{\varepsilon}

\def\l{\lambda} 

\def\m{\mu}

\def\n{\nu}

\def\s{\sigma} 
\def\t{\tau}
\def\f{\varphi}  \def\F{\Phi}
\def\th{\theta} 
\def\om{\omega}

\def\id{\hbox{id}}
\def\ker{{\rm Ker}}

\newtheorem{thm}{Theorem}[section]
\newtheorem{lem}[thm]{Lemma}
\newtheorem{defin}[thm]{Definition}

\newtheorem{cor}[thm]{Corollary}
\newtheorem{prop}[thm]{Proposition}
\newtheorem{rem}[thm]{Remark}
\newtheorem{conj}[thm]{Conjecture}
\theoremstyle{definition}
\newtheorem{examp}{Example}[section]

\def\dist{\mathop{\rm dist}}

\def\di{{\rm d}}

\def\idd{{1}\!\!{\rm I}}

\DeclareMathAlphabet{\mathpzc}{OT1}{pzc}{m}{it}

\begin{document}
\title[spectral properties and decoherence]
{Spectral and ergodic properties of completely positive maps and decoherence}
\author[F. Fidaleo]{Francesco Fidaleo}
\address{Francesco Fidaleo\\
Dipartimento di Matematica \\
Universit\`{a} di Roma Tor Vergata\\
Via della Ricerca Scientifica 1, Roma 00133, Italy} \email{{\tt
fidaleo@mat.uniroma2.it}}
\author[F. Ottomano]{Federico Ottomano}
\address{Federico Ottomano\\
Department of Computer Science\\
University of Liverpool\\
Ashton Street, Liverpool, L69 3BX, United Kingdom}\email{{\tt
Federico.Ottomano@liverpool.ac.uk}}
\author[S. Rossi]{Stefano Rossi}
\address{Stefano Rossi\\
Dipartimento di Matematica \\
Universit\`{a} degli studi di Bari Aldo Moro\\
Via Edoardo Orabona  4, Bari 70125, Italy} \email{{\tt
stefano.rossi@uniba.it}}
%\date{\today}

\begin{abstract}
\vskip0.1cm\noindent
In an attempt to propose more general conditions for decoherence to occur,
 we study spectral and ergodic properties of unital, completely positive maps on not necessarily unital $C^*$-algebras, with a particular focus on gapped maps for which the transient portion of the arising dynamical system can be separated from the persistent one. After some general results, we first devote our attention  to the abelian case by investigating the unital $*$-endomorphisms of, in general non-unital, $C^*$-algebras, and their spectral structure. The finite dimensional case is also investigated in detail, and examples are provided of unital completely positive maps for which the persistent part of the associated dynamical system is equipped with the new product making it a $C^*$-algebra, and the map under consideration restricts to a unital $*$-automorphism for this new $C^*$-structure, thus generating a conservative dynamics on that persistent part.\\
 
 
 \end{abstract}
\maketitle

\section{Introduction}
\emph{Quantum decoherence induced by the environment} is a rather suggestive way to give a dynamical explanation to the collapse
of the wave function  caused by a measurement process. Roughly speaking, for a dynamical system it simply means the possibility to separate the persistent part that is assumed to be a $C^*$-subalgebra of the original system on which the dynamics restricts to a conservative (or Hamiltonian) one, from the transient dynamics which disappears after the repeated action of the dynamics.

Without even attempting an exhaustive explanation of the physics involved, we rather limit ourselves to recalling  how quantum decoherence is usually framed in a purely mathematical context.
In the definition introduced by Blanchard and Olkiewicz (see
\cite{BO, O1, O2} for details), the time evolution is axiomatised in the following way.
The dynamics is encoded by a quantum Markov semigroup acting on a
von Neumann algebra $\mathcal{M}$ acting on the (separable) Hilbert space $\ch$, which is often $\mathcal{B}(\ch)$, the algebra of all bounded linear operators
acting on $\ch$. More explicitly, one has a family
$\{T_t: t\geq 0\}$ of normal completely positive unital maps on $\mathcal{M}$ satisfying  the semigroup property,
that is $T_0=I$ and $T_{t+s}=T_t\circ T_s$, for any $t, s\geq 0$.
The occurence of the decoherence then corresponds to having a decomposition of
$\mathcal{M}$ into a direct sum of the form
$\mathcal{M}=\mathcal{N}(T)\bigoplus \mathcal{M}_o$,
where $\mathcal{N}(T)$ is by definition the subalgebra
$$
\{x\in\mathcal{M}\mid T_t(x^*x)=T_t(x^*)T_t(x),\,T_t(xx^*)=T_t(x)T_t(x^*)\,\quad t\geq0\}
$$
often referred to as the multiplicative domain of $\{T_t: t\geq 0\}$,
and $\mathcal{M}_o:=\{x\in\mathcal{M}\mid\lim_{t\rightarrow\infty}\|T_t x\|=0\}$.

In \cite{CSU}, various necessary and sufficient conditions for decoherence to occur are given when
$\mathcal{M}=\mathcal{B}({\ch})$, for a finite dimensional Hilbert space $\ch$. Some relevant cases when $\ch$ is a separable infinite dimensional Hilbert space are also investigated, for the reader is referred to \cite{FSU} and the references cited therein.

One of the conditions for the decoherence is that the equality $\mathcal{N}(T)=V_o$ holds where $V_o$, which represents the persistent part under the evolution generated by the semigroup, is by definition the linear span  of the set
$\{x\in \mathcal{B}(\ch)\mid\mathcal{L}x=i\lambda x, \lambda\in\mathbb{R}\}$, with $\mathcal{L}$ being the generator of the semigroup.
As shown in \cite{CSU} ({\it e.g.} Section \ref{deco}, Example \ref{papy}), the inclusion $\mathcal{N}(T)\subset V_o$  may in general be strict, and
$V_o$ may even fail to be a subalgebra, in which case decoherence does not occur according to the usual definition. On the other hand, in such an example $V_o$ corresponds to the subspace of the invariant vectors. In addition, it is possible to define a new product in a natural way ({\it e.g.} \cite{K}) such that $T\lceil_{V_o}$ generates on $V_o$, equipped with this new $C^*$-product, a (trivial) conservative dynamics. Therefore, we feel it is interesting to point out that one can have cases where $V_o$ is not a subalgebra, but can nevertheless
be endowed with a new product that makes it into a $C^*$-algebra, see Section \ref{examples} for simple but nontrivial examples.

This raises the natural problem to decide under what conditions it is possible to redefine a product 
on $V_o$ which makes it into a $C^*$-algebra and which is compatible with the semigroup in
the sense that the semigroup itself now acts through $*$-automorphisms on $V_o$ with respect to the
new product.
The question can of course be asked for a single completely positive unital map without loosing anything concerning the conceptual aspects.

We address the question of the decoherence in a easier, and possibly more general form as follows. We start with a unital completely positive map $\Phi$ on a $C^*$-algebra 
$\ga$, and are seeking for a decomposition of $\ga$ into a topological direct sum of the form
$\ga_\infty\bigoplus\ga_o$ where $\ga_\infty$ is not necessarily a $C^*$-algebra, and
$\lim_{n\rightarrow\infty}\|\Phi^n(a)\|=0$ for any $a\in\ga_o$. Notice that the last is a strong ergodic property allowing the separation between transient and persistent part. 

Example \ref{exop} tells us that weaker conditions than the topological direct sum might be used to separate the persistent part from the transient one. Such conditions are still not well understood.

Concerning  some relevant ergodic properties associated to the peripheral spectrum, the reader is referred to \cite{F} and the reference cited therein.
Very recent developments of ergodic theory in the quantum setting, instead, can be found in 
\cite{Anzai1, Anzai2}.

Intuitively, the component $\ga_\infty$ representing the persistent part, is associated with the peripheral spectrum of $\Phi$,  $\sigma_{\rm ph}(\Phi):=\sigma(\Phi)\bigcap\mathbb{T}$, whereas
$\ga_o$, the transient part, is associated with the remaining part of the spectrum and does disappear after the repeated action of the dynamics.

It is then natural to see if $\ga_\infty$ can be endowed with a new product that turns it into a $C^*$-algebra.
In addition, it would also be desirable that
the restriction of $\Phi$ to $\ga_\infty$ should be a $*$-automorphism. 
Note that, at least in the finite dimensional case, when this is the case the isometry condition
$\|\F(a)\|=\|a\|$, $a\in\ga_\infty$ becomes a necessary condition for a splitting as above to exist, see Remark \ref{ism}.

In its full generality, though, this problem is much more complicated than it appears ({\it e.g.} Example \ref{exop}), and it is
beyond the scope of the present work.
Here, we limit ourselves to discussing the problem for gapped unital completely positive maps,
providing  some elementary cases for which the question is settled in the positive.

Among those, there is the case arising from finite dimensional dynamical systems satisfying all requirements. Indeed, in Section \ref{deco} we point out that, in all finite dimensional cases which are of course gapped, 
the persistent part, that is the direct sum of the eigenspaces of $\Phi$ corresponding to the peripheral eigenvalues,
can always be equipped with a new product that turns it into a $C^*$-algebra, in much the same spirit as the Choi-Effros
result for injective operator systems in \cite{CE}. In the general case, which at present is
out of the reach of our analysis, it might be possible that one can only recover
a Jordan algebra structure, for the Jordan structure of the selfadjoint part of a $C^*$-algebra is
actually all that matters for the description of the observables of a given quantum system, see {\it e.g.} \cite{EM}.

Since the approach we start developing in this paper
is mainly based on spectral theory,  the first thing to do is to study the spectral properties of
completely positive maps, the last being however a formidable task in the full generality. Section \ref{spbma} is indeed devoted to some basic properties about the spectrum of bounded maps acting on Banach which are an interest in itself.

Carrying out this programme, however, requires taking a step back to $*$-endomorphisms of
$C^*$-algebras, whose general spectral properties have not been addressed systematically, as far as we know, apart from
the commutative case, which is completely known in the unital case as an application of a general result due to
Sheffold in \cite{S}, see also \cite{W}. For such a purpose, at the end of Section \ref{sec2} we introduce a suitable generalisation of unitality for $*$-endomorphisms
between possibly non unital $C^*$-algebras which has a self-containing interest. In Section \ref{commucase}, we first describe how unital $*$-endomorphisms of abelian, not necessarily unital, $C^*$-algebras, can be described in terms of proper maps acting on the (locally compact) spectrum. This result enables us to
show that the spectrum of
a $*$-endomorphism is still of the type allowed by Sheffold's result on condition that the $*$-endomorphism is unital in our sense.

The paper ends with a section collecting some simple but illustrative examples for which the decoherence, according the approach followed in the present paper, occurs.

\section{Notations and preliminaries}\label{sec2}

We shortly recall some notions useful in the sequel, starting from the definition $\bt:=\{\l\in\bc\mid|\l|=1\}=\{e^{\imath\th}\mid \th\in[0,2\pi)\}$ of the unit circle of the complex plane. The closed unit disk is defined as
$\bd:=\{\l\in\bc\mid|\l|\leq1\}$. Obviously, $\partial\bd=\bt$.

For the normed space $\cx$, always complex, by $\cb(\cx)$ we denote the Banach algebra of all linear bounded operators acting on it. Without any further mention, we deal with complete normed spaces $\cx$ always different from $\{0\}$ if it is not otherwise specified, and with linear operators.

If $T\in\cb(\cx)$, we denote by $\s(T)\subset\bc$ the spectrum of $T$, which is a non empty compact subset of the complex numbers. The complement ${\rm P}(T):=\sigma(T)^c$ is the resolvent set, and 
$$
{\rm P}(T)\ni\l\mapsto R_T(\l):=(\l I-T)^{-1}\in\cb(\cx)
$$
the resolvent. The {\it spectral radius} ${\rm spr}(T)$ is defined by
$\max_{\l\in\s(T)}|\l|=\overline{\lim}_n\|T^n\|^{1/n}$, and satisfies ${\rm spr}(T)=\lim_n\|T^n\|^{1/n}$, see
{\it e.g.} \cite{Y}. One has
$0\leq{\rm spr}(T)\leq\|T\|$.
If $T\in\cb(\ch)$ is a normal operator acting on the Hilbert space $\ch$, then ${\rm spr}(T)=\|T\|$, but there are important classes of operators for which ${\rm spr}(T)<\|T\|$ already in the finite dimensional case. Indeed, if $T$ is a (generalised) nilpotent operator, $0={\rm spr}(T)<\|T\|$.
By $\sigma_{\rm ph}(T)$ we denote the peripheral spectrum of $T$, that is 
$$
\sigma_{\rm ph}(T):=\{\lambda\in\sigma(T): |\lambda|={\rm spr}(T)\}\,.
$$

If $\s_{\rm ph}(T)$ is separated by the remaining part of the spectrum ({\it i.e.} $\dist\big(\s_{\rm ph}(T),\s(T)\smallsetminus\s_{\rm ph}(T)\big)>0$), we say that $T$ has 
a {\it  mass gap} or it is {\it gapped}, where the last terminology is inherited from physics. Notice that if $\cx$ is finite dimensional, then each linear operator $T:\cx\to\cx$ is gapped.

The case of interest in this paper is when ${\rm spr}(T)\leq 1$. In such a situation, which is tacitly assumed henceforth, we continue to say that $T$ is  gapped whenever 
${\rm spr}(T)<1$. 

For the convenience of the reader, we also recall the definition of the {\it approximate point spectrum}, useful in the sequel. 
\begin{defin}
The approximate point spectrum of $T\in\cb(\cx)$ is the set $\s_{\rm ap}(T)$
of those $\lambda\in\bc$ for which there exists a sequence $(x_n)_{n\in\bn}\subset\cx$ with $\|x_n\|=1$, such that 
$\lim_{n}\|(T-\lambda I)x_n\|=0$.
\end{defin}

\smallskip

In the sequel, the Banach space $\cx$ will often be a 
$C^*$-algebra $\ga$, unital with identity $\idd\equiv I:=\idd_\ga$, but also the non unital case is treated in some detail. The maps $\F\in\cb(\ga)$ we deal with in the sequel are mainly 
unital completely positive linear maps where, for non unital  $C^*$-algebras and $*$-endomorphisms, we prove that a reasonable definition can be provided by using approximate unities.

We recall that, for maps $\F\in\cb(\ga)$ \emph{positive} means $\F(a^*a)\geq 0$ for any $a\in\ga$, whereas $\F$ is {\it completely positive}
if the map 
$$
\F_n:=\F\otimes\id_n:\bm_n(\ga)\to\bm_n(\ga)
$$ 
is positive for every $n\geq 1$. Recall that $\F_n$ is simply defined by $\F$ acting on the matrix $a\in \bm_n(\ga)$ as
$$
\F_n(a)_{ij}=\F(a_{ij})\,,\quad i,j=1,\dots,n\,.
$$

Notice that, for completely positive unital maps acting on unital $C^*$-algebras, $\|\F\|=\|\F(I)\|=1$, and thus ${\rm spr}(\F)=\|\F\|=1$. Similar results might be true for general completely positive maps between, in general non unital, $C^*$-algebras.
The $*$-homomorphisms between $C^*$-algebras, and in particular $*$-automorphisms, are all examples of completely positive maps.

An {\it operator system} is a
subspace $\mathcal{S}$ of a $C^*$-algebra $\ga$ which is stable under taking adjoints, that is $\cs^*\subset\cs$ (and therefore $\cs^*=\cs$), but is not necessarily norm closed. It is also supposed that $\idd_\ga\in\cs$, provided $\ga$ 
is unital.\footnote{In alternative, one might suppose that $\cs$ contains some approximate unity $(u_\iota)_\iota$ of $\ga$.}

\smallskip

We now extend the notion of  unital $*$-endomorphism for, possibly non unital, $C^*$-algebras. 
\begin{prop}
Let $\ga$ be a $C^*$-algebra with an approximate unity $(u_\iota)_{\iota\in I}\subset\ga$, and $\F:\ga\to\ga$ be a $*$-endomorphism.

Suppose that $\big(\F(u_\iota)\big)_{\iota\in I}$ is also an approximate unity for $\ga$. 
Then 
$\big(\F(v_\n)\big)_{\n\in N}$ is an approximate unity for $\ga$ for each approximate unity $(v_\n)_{\n\in N}$ of $\ga$.
\end{prop}
\begin{proof}
It is enough to show that $\F(v_\n)a\to a$ for each $a\in\ga$, the limit of the product $a\F(v_\n)\to a$ being similar. 

Since $\big(\F(u_\iota)\big)_{\iota\in I}$ is supposed to be an approximate unity, for fixed $\eps>0$ there exist $\iota_o\in I$ such that
$\|\F(u_{\iota_o})a-a\|<\eps/2$. Then we have
\begin{align*}
\|\F(v_{\a})a-a\|\leq&\|\F(v_{\a})a-\F(v_{\a})\F(u_{\iota_o})a\|
+\|\F(v_{\a}u_{\iota_o})a-\F(u_{\iota_o})a\|\\
+&\|\F(u_{\iota_o})a-a\|
\leq\eps+\|a\|\|v_{\a}u_{\iota_o}-u_{\iota_o}\|\to\eps\,.
\end{align*}
The proof now follows as $\eps>0$ is arbitrary.
\end{proof}
The previous result justifies the following
\begin{defin}
\label{apu}
Let $\F$ be a $*$-endomorphism of the $C^*$-algebra $\ga$. It is said to be unital if $\big(\F(u_\iota)\big)_{\iota\in I}$ is an approximate unity for $\ga$ for some (and hence for all) approximate unity $(u_\iota)_{\iota\in I}$ of $\ga$.
\end{defin}
Suppose that $\F$ is a unital positive map acting on the unital $C^*$-algebra $\ga$. The pair $(\ga,\F)$ provides a (discrete) {\it $C^*$-dynamical system} in a natural way.\footnote{For the non unital $C^*$-algebras, the same holds true if $\F$ is a unital $*$-endomorphism. Dynamical systems for which the completely positive map $\F$ is not unital could be also considered.} It is said to be {\it Hamiltonian or conservative} if $\F$ is a $*$-automorphism of $\ga$, and {\it dissipative} otherwise. 

Given a strongly continuous ($C_o$ for short) one-parameter semigroup
$\{T(t): t\geq 0\}$ of completely positive maps,  for each $\d>0$ (which can be chosen to be 1 without harming generality) we obtain a discrete $C^*$-dynamical system $(\ga,\F_{A,\d})$ where 
$\F_{A,\d}:=T(\delta)$.  
Thanks to the Hille-Yoshida Theorem ({\it e.g.} \cite{Y}), the semigroup has a generator, that is $T(t)=e^{-tA}$, where $A$ is a closed densely defined  operator acting on $\ga$ with
$$
\s(A)\subset\{\l\in\bc\,|\,{\rm Re}(\l)\geq0\}\,.
$$

Since $\F_{A,\d}=e^{-\delta A}$, its peripheral spectrum corresponds to
 $$
\s(A)\bigcap\{\l\in\bc\,|\,{\rm Re}(\l)=0\}\,,
$$
and the gapped case occurs precisely when 
$$
\dist\big(\s(A)\bigcap\{\l\in\bc\,|\,{\rm Re}(\l)>0\},\{\l\in\bc\,|\,{\rm Re}(\l)=0\}\big)>0\,.
$$
In brief, $\s(\F_{A,\d})\subset\bd$ and $A$ is gapped if and only if the $\F_{A,\d}$ are gapped. 

The above considerations about the semigroups clearly explain why $\s(\F)\subset\bt$ corresponds to the Hamiltonian case. 

\medskip

We end with the following well-known notion of classical ergodic theory. Indeed, let $X$ be a locally compact, possibly non compact, Hausdorff space, and $\t:X\to X$ a proper map. The unital $*$-endomorphism $\F:C_o(X)\to C_o(X)$ defined by $\F(f):=f\circ\t$, $f\in C_o(X)$, is named
{\it the Koopman operator} associated to $\t$, see Section \ref{commucase} for details.

\section{Some spectral properties of bounded linear maps}
\label{spbma}

In the present section we collect some fundamental spectral properties of bounded linear maps, perhaps known to the experts, which will be useful in the sequel.
 
 \medskip
 
Let $T\in\cb(\cx)$ with ${\rm spr}(T)\leq 1$ be a gapped operator. Then there is a Jordan curve $\g\in {\rm P}(T)$ surronding the part of the spectrum inside $\stackrel{o}{\bd}$, corresponding to the whole spectrum if ${\rm spr}(T)<1$. Therefore, we can define a continuous projection $Q\in\cb(\cx)$, which does not depend on the chosen curve satisfying the conditions required above, as
\begin{equation}
\label{gaap}
Q:=\frac1{2\pi\imath}\rcircleleftint_\g R_T(\l)\di\l\,,\quad P:=I-Q\,.
\end{equation}
We get that $\cx=P\cx\bigoplus Q\cx$ is the topological inner direct sum of $P\cx$ and $Q\cx$. Obviously, if ${\rm spr}(T)<1$ then $Q=I$ and thus $P=0$.
Conversely, if ${\rm spr}(T)=1$ and $\s(T)=\s_{\rm ph}(T)$, then $Q=0$ and thus $P=I$.
\begin{prop}
\label{ergorisultato}
Let $T\in\cb(\cx)$ be a gapped operator with ${\rm spr}(T)\leq1$. Then $\lim_{n\rightarrow +\infty} T^nQ=0$ in norm.
\end{prop}
\begin{proof}
We start by showing the following identity
\begin{equation}
\label{isd}
T^nQ = \frac1{2\pi\imath}\rcircleleftint_\g \l^nR_T(\l)\di\l\,,
\end{equation}
which holds for any natural number $n$. We will argue by induction on $n$. As the Jordan curve $\g$, we can choose the circle centred at $0$ of radius $r$, for some $r<1$.

We recall the following well-known identity
$$
TR_T(\l)=\l R_T(\l)-I\,,
$$
holding on ${\rm P}(T)$ by holomorphic functional calculus. 
If $n=0$, we just have \eqref{gaap}. Suppose that \eqref{isd} holds true for some integer $n$. We get
\begin{align*}
T^{n+1} Q = &T(T^n Q) = \frac1{2\pi\imath}\rcircleleftint_\g \l^nTR_T(\l) \di\l \\
= &\frac1{2\pi\imath}\rcircleleftint_\g\big(\l^{n+1}R_T(\l)-\l^nI\big)\di\l\\
=&\frac1{2\pi\imath}\rcircleleftint_\g\l^{n+1}R_T(\l)\di\l\,,
\end{align*}
where the last equality follows by Cauchy's formula.
The conclusion can now be easily achieved, for we have the estimate
\begin{align*}
\|T^nQ\| &\le \frac{1}{2 \pi}\int_{\gamma}\|\l^n R_T(\l)\| \di\l
 \leq |r|^{n+1}\max_{|\l|= r} \|R_T(\l)\| \xrightarrow[n \to +\infty]{}0\,.
\end{align*}
\end{proof}
For gapped operators $T\in\cb(\cx)$, as soon as we have selected the elements $x\in\cx$ corresponding to the "spectral subspace" associated to $\s(T)\bigcap\stackrel{o}{\bd}$ ({\it cf.} 
Proposition \ref{ergorisultato}), it is very natural to address the same problem for non-gapped operators. In order to avoid trivial cases, we first analyse the situation for which $T$ is gapped and ${\rm spr}(T)=1$. For the spectrum, we have 
$$
\s(T)=\s_{\rm ph}(T)\bigcup\big(\s(T)\bigcap\stackrel{o}{\bd}\big)\,,
$$
and in addition $\s_{\rm ph}(T)\neq\emptyset$. Therefore, for $T_P:=TP$, $\s(T_P)=\s_{\rm ph}(T)$ and $T_P$ is invertible. Taking into account that, if $x=Px$ then 
$(T_P)^{-n}x=T^{-n}x$ (where $T^{-n}$ is defined via its own domain), we easily get for $x=Px$,
$$
R_T(\l)x=-\sum_{n=0}^{+\infty}\l^nT^{-(n+1)}x\,,
$$
and the power-series converges at least for $\l<1$. Therefore, we conclude that 
$$
x\in{\rm Ran}(P)\Rightarrow\overline{\lim}_n\|T^{-n}x\|^{1/n}\leq1\,.
$$

By passing to the general situation of $T\in\cb(\cx)$, with the unique additional condition that $0\in{\rm P}(T)$, we obtain in analogy to the above considerations,
\begin{equation}
\label{pse}
R_T(\l)x=-\sum_{n=0}^{+\infty}\l^nT^{-(n+1)}x\,,
\end{equation}
and the series in \eqref{pse} converges at least in the disk $\{\l\in\bc\mid|\l|<\dist(0,\s(T))\}$.

For $r\geq\dist(0,\s(T))$, consider the circle $C_r:=\{\l\in\bc\mid|\l|=r\}$ together with 
$$
\cx_{T,r}:=\{x\in\cx\mid\overline{\lim}_n\|T^{-n}x\|^{1/n}\leq 1/r\} \,.
$$

It is easy to show that $\cx_{T,r}\subset\cx$ is a closed linear subspace. Indeed, it is closed by multiplication by the scalars. In addition, $x,y\in\cx_{T,r}$,
$R_T(\l)(x+y)=R_T(\l)x+R_T(\l)y$ and the series describing $R_T(\l)(x+y)$ converges at least 
$$
|\l|<\min\bigg\{\frac1{\overline{\lim}_n\|T^{-n}x\|^{1/n}},\frac1{\overline{\lim}_n\|T^{-n}y\|^{1/n}}\bigg\}\,.
$$
Therefore,
$$
\frac1{\overline{\lim}_n\|T^{-n}(x+y)\|^{1/n}}\geq\min\bigg\{\frac1{\overline{\lim}_n\|T^{-n}x\|^{1/n}},\frac1{\overline{\lim}_n\|T^{-n}y\|^{1/n}}\bigg\}\,,
$$
which leads to 
$$
\overline{\lim}_n\|T^{-n}(x+y)\|^{1/n}
\leq\max\bigg\{\overline{\lim}_n\|T^{-n}x\|^{1/n},\overline{\lim}_n\|T^{-n}y\|^{1/n}\bigg\}\leq1/r\,.
$$
Suppose that $(x_k)_{k\in\bn}\subset\cx_{T,r}$, and $x_k\to x$. By exchanging the symbols of the sum and the limit, we see that $R_T(x)$ converges at least for
$|\l|<\inf_n\big\{1/\overline{\lim}_n\|T^{-n}(x+y)\|^{1/n}\big\}$
which, as before, leads to $\overline{\lim}_n\|T^{-n}x\|^{1/n}\leq1/r$. Therefore, $x\in\cx_{T,r}$.

Now we consider the case of interest, that is a not necessarily gapped $T\in\cb(\cx)$, such that ${\rm spr}(T)=1$ and $0\in{\rm P}(T)$. The above considerations suggest that 
the closed subspace $\cx_{T,1}$ might characterise the "spectral subspace" associated to $\s_{\rm ph}(T)$. Certainly, $0\in\cx_{T,1}$, but it is still unclear under what conditions, which  should be as minimal as possible, $\cx_{T,1}$ does describe the spectral subspace associated with $\s_{\rm ph}(T)\neq\emptyset$.

Another possible characterisation of the spectral subspace we are searching for the cases of interests in the present paper, that is when $\F\in\cb(\ga)$ is a unital completely positive map acting on the $C^*$-algebra $\ga$, is $\|\F(x)\|=\|x\|$, see Remark \ref{ism}.

\medskip

We now review the main properties of the approximate point spectrum.
\begin{prop}
\label{approximatepoint}
If $T\in\cb(\cx)$, then:
\begin{itemize}
\item[(i)] $\sigma_{\rm ap}(T)$ is a non empty closed subset of $\sigma(T)$,
\item[(ii)] $\partial \sigma(T) \subset \sigma_{\rm ap}(T)$.
\end{itemize}
\end{prop}
\begin{proof}
(i) We start by noticing that $\sigma_{\rm ap}(T)\neq\emptyset$ easily follows by (ii). If now $\lambda \notin \sigma(T)$, then $T-\lambda I$ admits a bounded inverse by the Bounded Inverse Theorem. In particular,
$$
\|(T-\lambda I)x\| \geq\frac{\|x\|}{\|(T-\lambda I)^{-1}\|}\,.
$$
This means that $T-\lambda I$ is bounded from below, and hence $\s_{\rm ap}(T) \subset \sigma(T)$. Fix now $\l_o\in\s_{\rm ap}(T)^{\rm c}$. We have as above,
$\|(T-\lambda_o I)x\| \geq\|x\|/\|(T-\lambda_o I)^{-1}\|$. Consider now the open ball $B$ centered in $\l_o$ of radius $\|(T-\lambda I_o)^{-1}\|^{-1}>0$, and fix $\l\in B$. We get
\begin{equation*}
\begin{split}
\|Tx-\lambda x\| &= \|Tx-(\lambda - \lambda_o)x - \lambda_o x\|\\&\geq \, \|Tx - \lambda_o x\| - |\lambda - \lambda_o|\|x\|\\ &\geq c\|x\|\,,
\end{split}
\end{equation*}
with $c:=\|(T-\lambda I_o)^{-1}\|^{-1}-|\lambda - \lambda_o|>0$.
Thus $\sigma_{\rm ap}(T)$ has an open complement and hence is closed.

(ii) Let $\lambda \in \partial \sigma(T)$ and $\varepsilon > 0$. Since $\sigma(T)$ and $\mathbb{C} \smallsetminus \sigma(T)$ have the same boundary, there is $\mu \notin \sigma(T)$ such that $|\lambda - \mu| < \frac{\varepsilon}{2}$. We have
\begin{equation*}
\frac{2}{\varepsilon} \le \frac{1}{\dist(\mu,\s(T))} \le \|(T-\mu I)^{-1}\| \,,
\end{equation*}
since it is a well-known fact that $\|R_T(\lambda)\| \rightarrow \infty$ as $\dist(\mu,\s(T)) \rightarrow 0$.
Therefore, there is a $x \in\cx$ such that $\|x\| = 1$ and 
$1/\varepsilon \leq \|(T-\mu I)^{-1}x\|$.

Let $y = \frac{1}{\|(T-\mu I)^{-1}x\|}(T-\mu I)^{-1}x$, then $\|y\| = 1$ and
$\|(T-\lambda I)y - (T- \mu I)y\| < \varepsilon/2$. It follows that 
\begin{align*}
\|(T-\lambda I)y\| &\le \|(T-\lambda I)y - (T- \mu I)y\| + \|(T-\mu I)y\|\\
& < \frac{\varepsilon}{2} + \frac{\|x\|}{\|(T-\mu I)^{-1}\|}
 < \frac{3\varepsilon}{2}\,,
\end{align*}
and hence it results $\lambda \in \sigma_{\rm ap}(T)$. 
\end{proof}
We are now in a position to prove a useful result concerning an isometry $T\in\cb(\cx)$, which is known to the experts but seems to be difficult to find in the literature.
The final part of the proof that we provide has actually been taken from \cite{SpectrumIsometry}.
All we do is add some few details on the missing initial part.
\begin{thm}
\label{isospectrum}
Let $\cx$ be a Banach space and $T \in \cb(\cx)$ an isometry. Then, either of the two conditions holds true:
\begin{itemize}
\item[(i)] if $T$ is surjective, then $\sigma(T) \subset \mathbb{T}$;
\item[(ii)] if $T$ is not surjective, then $\sigma(T)=\mathbb{D}$.
\end{itemize}
\end{thm}
\begin{proof}
We begin to handle the first case. First observe that the spectrum of
$T$ is contained in $\mathbb{D}$ as $\|T\|=1$.
Now, if there were $\lambda \in \sigma(T)$ with $|\lambda| < 1$,  there would exist $\mu \in \sigma(T^{-1})$ with $|\mu| > 1$ given that $\sigma(T^{-1})= \frac{1}{\sigma(T)}$. This would imply
that the spectral radius of $T^{-1}$ should be strictly greater than $1$, which is impossible because $T^{-1}$ is still an isometry.

For the second case, if $\s(T)\subsetneq\bd$, then it would exist $\l\in\stackrel{o}{\bd}\bigcap\partial\s(T)$. 
By Proposition (\ref{approximatepoint}), $\l\in\sigma_{\rm ap}(T)$, and thus there exists a sequence $(x_i)_i \subset\cx$ such that $\|x_i\| = 1$ for all $i$, and $\|Tx_i - \lambda x_i\| \rightarrow 0$ as $i \rightarrow \infty$. Since $T$ is an isometry, $\|x_i\|=1$ for all $i$, and $|\l|<1$, 
$$
\|Tx_i - \lambda x_i\| \geq \|Tx_i\|-|\lambda|\|x_i\| = \|x_i\| - |\lambda| \|x_i\|=1-|\l|>0\,,\quad i\in\bn\,,
$$
which is a contradiction.
\end{proof}
For an isometry $V$ acting on the Hilbert space $\ch$, the same conclusion can be reached by a simple application of \textit{Wold's decomposition}, which states that 
$$
\ch\cong\ch_o\bigoplus_{\iota\in I}\ell^2(\bn)\,,\quad V\cong U_o\bigoplus_{\iota\in I} S\,,
$$
for some index-set $I$. Here, $S$ is the forward shift acting on $\ell^2(\bn)$, and $U_o$ is a unitary operator ({\it i.e.} $U_oU^*_o=\idd_{\ch_o}=U^*_oU_o$) acting on the Hilbert 
space $\ch_o$. The trivial case corresponds to $\ch_o=\{0\}$ and $I=\emptyset$. The nontrivial unitary case corresponds to $\ch_o\neq\{0\}$ and $I=\emptyset$, in which case 
$\s(V)\subset\bt$. The non-unitary case corresponds to $\iota_o\in I\neq\emptyset$, in which case $\bd=\s(S)\subset\s(V)\subset\bd$.

Since injective $*$-homomorphisms between $C^*$-algebras are automatically isometric, the following result
is a straightforward application of Theorem \ref{isospectrum}. 

A $C^*$-algebra is {\it simple} if it contains no closed two-sided ideals apart from ${0}$ and itself. Since for a $*$-endomorphism $\F$, $\ker(\F)$ is a closed two-sided ideal, we deduce that a  (non null) $*$-endomorphism of a simple $C^*$-algebra is always injective and thus an isometry, see {\it e.g.} \cite{T}.
\begin{cor}
\label{sime}
If $\Phi$ is an injective $*$-endomorphism of a $C^*$-algebra $\ga$, then $\sigma(\F)$ is either contained
in $\bt$, which happens when $\F$ is surjective, equivalently when $\F$ is an automorphism of $\ga$, or $\sigma(\Phi)=\mathbb{D}$.

Therefore, if $\ga$ is simple and $\F$ is a nontrivial $*$-endomorphism, we have:
\begin{itemize} 
\item[(i)] $\s(\F)=\bd\iff{\rm Ran}(\F)\neq\ga$,
\item[(ii)] $\s(\F)\subset\bt\iff{\rm Ran}(\F)=\ga$.
\end{itemize}
\end{cor}
\begin{rem}
\label{rut}
If a $*$-endomorphism $\Phi$ of a $C^*$-algebra $\ga$ is not surjective, then $0\in\s(\F)$. If in addition it is injective, it is an isometry and thus $\s(\F)=\bd$. If $\F$ is not unital, then $\F$ is not surjective, and thus $\s(\F)=\bd$ whenever it is injective.
\end{rem}
\begin{proof}
By Corollary \ref{sime}, we only have to show that if $\F$ is not unital ({\it cf.} Definition \ref{apu}) then it cannot be surjective. Indeed,
if $(u_\iota)_{\iota\in I}\in\ga$ is any approximate unity, then  $\big(\F(u_\iota)\big)_{\iota\in I}$ is no longer an approximate
unity. Then there exists some $b\in\ga$ such that either $\big(\F(u_\iota)b\big)_{\iota\in I}$ or
$\big(b\F(u_\iota)\big)_{\iota\in I}$ does not converge to $b$. We only deal with the former case, 
 the latter being completely analogous.

Indeed, suppose that $\F$ is surjective, and thus there exists $a\in\ga$ such that $\F(a)=b$. We get
\begin{align*}
b=&\F(a)=\F\big(\lim_\iota(u_\iota a)\big)=\lim_\iota\F(u_\iota a)\\
=&\lim_\iota\big(\F(u_\iota)\F(a)\big)=\lim_\iota\big(\F(u_\iota)b\big)\,,
\end{align*}
which is a contradiction.
\end{proof}
The problem of determining what the spectrum of an endomorphism of
a $C^*$-algebra looks like is a very interesting one which, to our knowledge, is far from being settled in its full generality.
However, there are two classes of $C^*$-algebras for which the problem can be given
a complete answer: unital $*$-endomorphisms of abelian $C^*$-algebras, and not necessarily unital $*$-endomorphisms of finite dimensional
$C^*$-algebras.
We are going to tackle the latter, deferring the former case to the next section.  
\begin{prop}
\label{findimendo}
Let $\ga$ be a finite dimensional $C^*$-algebra, and $\F$ a, not necessarily unital, $*$-endomorphism. Then $\s(\F)\subset\{0\}\bigcup\bt$.
\end{prop}
\begin{proof}
Suppose that $\l\in\s(\F)$ with $0<|\l|<1$, and thus $\F(b)=\l b$ for some $b\neq0$. With $\m=|\l|^2$, we have that $a_n:=(b^*b)^n$ is a non null eigenvector of $\F$ corresponding to the eigenvalue $\m^n$, $n=1,2\dots\,\,\,$. Since $|\s(\F)|<+\infty$, the set $\{\m^n\mid n=1,2\dots\}$ is finite, which is possible only if $|\l|\in\{0,1\}$.
\end{proof}

\section{Commutative case}
\label{commucase}

The present section is devoted to $*$-endomorphisms $\F:\ga\to\ga$ of abelian, not necessarily unital, $C^*$-algebras $\ga$. If $\F$ is not unital, some properties concerning $\s(\F)$ are covered by Remark \ref{rut}. Therefore, from now on we restrict the analysis to the unital case.

For the abelian $C^*$-algebra $\ga$, if $X:=\s(\ga)$, $X$ is a locally compact Hausdorff space and $\ga\sim C_o(X)$, with the convention that if $X$ is already compact, then $X_\infty=X\bigcup\emptyset$ and $C_o(X)=C(X)$.\footnote{ With $\sigma(\mathfrak{A})$ we denote the \emph{Gel'fand spectrum}
of the abelian $C^*$-algebra $\mathfrak{A}$, see {\it e.g.} \cite{T}, Section I.3.}
\begin{thm}
\label{kooo}
Let $\F$ be a unital $*$-endomorphism of the abelian $C^*$-algebra $\ga\sim C_o\big(\s(\ga)\big)$. Then there exists a uniquely determined proper map $\t:\s(\ga)\to\s(\ga)$ such that 
$\F$ is the Koopman operator of $\tau$:
$$
\F(f)=f\circ\t\,,\quad f\in C_o\big(\s(\ga)\big)\,.
$$
\end{thm}
\begin{proof}
We start by noticing that the proper map defining $\F$ is uniquely determined because $C_o(X)$ separates the points of $X=\sigma(\mathfrak{A})$. We also note that, if $X$ is already compact, then
$\t:=\F^{\rm t}\lceil_X$, the restriction of the transpose of $\F$ to the set of characters, does the job.

Suppose that $X$ is not compact, and consider the unital $C^*$-algebra $\ga_{\idd}:=\ga+\l\idd$ obtained from $\ga$ after adjoining a unity. Obviously, $\ga_{\idd}=C(X_\infty)$. Let 
$\F_{\idd}$ the natural extension of $\F$ to $\ga_{\idd}$ given by $\F_{\idd}(f+\l\idd):=\F(f)+\l$, which is a unital $*$-endomorphism by construction. As before, we put 
$\t_{\idd}:=(\F_{\idd})^{\rm t}\lceil_{X_\infty}$, which is a continuous map of $X_\infty$ realising $\F_{\idd}$: $\F_{\idd}(f)=f\circ\t_{\idd}$, for every $f\in C(X_\infty)$.

If $\f\in X$ is a character, the corresponding character $\f_{\idd}\in X_\infty$ is defined as before by $\f_{\idd}(f+\l\idd):=\f(f)+\l$, and the remaining character $\f_\infty$ in $X_\infty$ corresponding to the point at infinity is defined  by $\f_{\infty}(f+\l\idd):=\l$. 

To end the proof, we have to show that $\t_{\idd}(X)\subset X$ and $\t_{\idd}(\f_\infty)=\f_\infty$, and therefore this imply that $\t:=\t_{\idd}\lceil_{X}$ will be the proper self-map of $X$ we are searching for. 
To this aim, we consider $\om=\f_{\idd}$ for some $\f\in X$. Then
$$
\t_{\idd}(\om)(f+\l\idd)=\f(\F(f))+\l\,,\quad f\in C_o(X)\,,\,\,\l\in\bc\,.
$$

Consider any approximate unity $(u_\iota)_{\iota\in I}$ of $\ga$, and suppose that $\f\circ\F=0$. Then $\f(\F(u_\iota))=0$ for each $\iota\in I$.
Since $\F$ is unital, $\big(\F(u_\iota)\big)_{\iota\in I}$ is also an approximate unity of $\ga$, and therefore
$$
\f(f)=\lim_\iota\f(f\F(u_\iota))=\lim_\iota\f(f)\f(\F(u_\iota))=\f(f)\lim_\iota\f(\F(u_\iota))=0\,,
$$
which leads to a contradiction.
The proof now follows because,  $\t_{\idd}(\f_\infty)=\f_\infty$ by straightforward computation.
\end{proof}
In the above theorem, the hypothesis of unitality cannot be dropped. Indeed, the following simple example explains what can happen in non unital cases.
\begin{examp}
Consider the real line $\br$, together with the function $\tau(x):=\ln|x|$, for $x\in\br\setminus\{0\}$. It does not extend to a continuous 
function on the whole real line since $\lim_{x\rightarrow 0}\tau(x)=-\infty$. However, it extends to a continuous function $\tilde\tau$ from the one-point compactification $\br_\infty$ 
of $\br$ to itself, by defining $\tilde\tau(0):=\infty$ and $\tilde\tau(\infty):=\infty$.

Notice that, if $f\in C_o(\br)$, then $f\circ\tilde\tau$ is still continuous on $\br$ and vanishes at infinity.
Indeed, if we denote by $\tilde f$ the continuous extension of $f$ to $\br_\infty$, that is $\tilde f(\infty)=0$,  we have $\tilde f\circ\tilde\tau(0)=0$ and 
$\lim_{x\rightarrow\infty}f(\tau(x))=\lim_{x\rightarrow\infty} f(\ln|x|)=0$.
In other words, the map $C_o(\br)\ni f \rightarrow f\circ\tilde\tau\in C_o(\br)$ defines a $*$-endomorphism
of $C_o(\br)$ that does not come from any map from $\br$ to itself.

This is not in contradiction with Theorem \ref{kooo} because this endomorphism fails to be unital. In fact, since $u_\iota(0)=0$ for any $\iota\in I$, for the approximate unit 
$(u_\iota)_{\iota\in I}$ of $C_o(\br)$, no approximate unit
can be sent to an approximate unit.
\end{examp}

The following result will come in useful in the sequel.\footnote{It is expected that this result could be true for unital $*$-endomorphisms of not necessarily abelian $C^*$-algebras.}
\begin{prop}
\label{issk}
Let $\F:\ga\to\ga$ be a unital $*$-endomorphism of an abelian $C^*$-algebra $\mathfrak{A}$, together with its natural extension $\F_{\idd}:\ga_{\idd}\to\ga_{\idd}$. Then $\s(\F)=\s(\F_{\idd})$.
\end{prop}
\begin{proof}
We restrict the situation to the nontrivial situation of non-unital algebras. We also note that $1\notin{\rm P}(\F_{\idd})$ because $\F_{\idd}$ is unital by construction.

Suppose first that $\l\in{\rm P}(\F_{\idd})$, and therefore $\l\neq1$. Take $f\in C_o(X)$, which corresponds a unique $g+\m\idd$, $g\in C_o(X)$, {\it i.e.} $\lambda(g+\mu\idd)-\Phi_{\idd}(g+\mu\idd)=f$, from which we arrive at the equality
 $f=\l g+\l\m-\F(g)-\m$. By evaluating it at $\infty$ we then get $0=(\l-1)\m$. Therefore, $\m=0$ and so $\l\in{\rm P}(\F)$.

Suppose now that $1\neq\l\in{\rm P}(\F)$. It is easy to show that
$$
R_{\F_{\idd}}(\l)(f+\m):=R_{\F}(\l)(f)+\frac{\m}{\l-1}
$$
is well defined and provides the resolvent $R_{\F_{\idd}}(\l)$.

It remains to check what happens for $\l=1$. If $1$ is a cluster point of $\s(\F_{\idd})$, thanks to the above part it is also a cluster point of $\s(\F)$ and thus $1\in\s(\F)$. Suppose that
$1$ is instead an isolated point of $\s(\F_{\idd})$.
Therefore, there exist nontrivial projections $P,Q\in\cb(\ga_{\idd})$ given by
$$
Q:=\frac1{2\pi\imath}\rcircleleftint_{\g}R_{\F_{\idd}}(\l)\di\l\,,\quad P:=I-Q\,,
$$
where $\g$ is a sufficiently small circle surronding $1$. 

We then deduce that $\ga_{\idd}$ is the topological direct sum $P\ga_{\idd}\bigoplus Q\ga_{\idd}$, and the summand $Q\ga_{\idd}$ contains $\idd$. The case $Q\ga_{\idd}=\bc\idd$ cannot occur because, otherwise, the point at infinity $\infty$ would be an isolated point of $\s(\ga_{\idd})$, contradicting that $\ga$ were not unital. Therefore, $Q\ga_{\idd}$ must contain a fixed-point $f\in\ga_{\idd}$ for $\F$ which is not a  constant function. But then $f-f(\infty)\idd\in\ga$ is a non-trivial fixed point
of $\F$, and thus $1\in\s(\F)$ as well.
\end{proof}
Notice that 1 is always in the approximate point spectrum of both spectra. In one case it is in the peripheral spectrum, whereas in the other case it may be or not an eigenvalue.

\medskip

We are now in a position to prove the main result of the present section concerning the spectrum of unital $*$-endomorphisms of, possibly non unital, abelian $C^*$-algebras. In fact, our result is an extension to
non-unital $C^*$-algebras of a consequence of a theorem by E. Scheffold  (same statement) 
concerning the spectrum of a positive linear operator acting  on a Banach lattice, see \cite{S, W, Gluck}.

\begin{thm}
Let $\F:\ga\to\ga$ be a unital $*$-endomorphism of the abelian $C^*$-algebra $\ga$, and $\t$ the 
corresponding Koopman operator. Then, either $\s(\F)=\bd$ or
$\stackrel{o}{\bd}\smallsetminus\{0\}\subset{\rm P}(\F)$. In particular, the following hold true:
\begin{itemize}
\item[(i)] if $\t(\s(\ga))\neq\s(\ga)$, then $0\in\s(\F)$;
\item[(i)] if for each $n\in\bn$, $\t^{n+1}(\s(\ga))\neq\t^n(\s(\ga))$, then $\s(\F)=\bd$;
\item[(iii)] if instead there exists $n_o\in\bn$ such that $\t^{n_o+1}(\s(\ga))=\t^{n_o}(\s(\ga))$, either $\s(\F)=\bd$ or $\s(\F)\subset\bt\bigcup\{0\}$.
\end{itemize}
\end{thm}
\begin{proof}
The proof follows by collecting Proposition \ref{issk}, Theorem 2.7 in \cite{S}, and Theorem \ref{kooo}. Indeed, first one notes that $\s(\F)$ coincides with that $\s(\F_{\idd})$, where $\F_{\idd}$ is the unital extension of $\F$ to $\ga_{\idd}$. Then one applies Sheffold's theorem to recover the structure of $\s(\F)$, and the connections between the properties of $\s(\F)=\s(\F_{\idd})$ and those of the Koopman map of $\F_{\idd}$. Finally, one notes that the properties of  $\tau$ and $\tau_{\idd}$ characterising the spectral properties are the same because $\t_{\idd}(\s(\ga))\subset\s(\ga)$, and $\t_{\idd}(\infty)=\infty$ where $\infty\equiv\f_\infty$ is the character at infinity.
\end{proof}

\section{Gapped completely positive maps and the decoherence}\label{deco}

The aim of the present section is to manage the problem of the decoherence by investigating in detail the spectral, and therefore the ergodic properties, of the completely positive map under consideration.  

More precisely, we start with a discrete $C^*$-dynamical system $(\ga,\F)$, where $\ga$ is a unital $C^*$-algebra and $\F$ a unital completely positive map $\F$. 
Therefore, from now on we tacitly assume the unitality of the involved algebras and maps.
In such a situation, $\s(\F)\subset\bd$ and ${\rm spr}(\F)=\|\F\|$.

The idea will be that of separating the "observables" associated to the peripheral part of the spectrum which are wandering and therefore constitute the part of the $C^*$-algebra which is persistent relative to the dynamics generated by $\F$. The remaining part is associated to the portion of the spectrum inside the unit disk which, hopefully, can represent the transient part. By Proposition \ref{ergorisultato}, the last property is certainly true whenever $\F$ is gapped. 

The first step is to investigate the structure of the elements associated to the peripheral part of the spectrum. For such a purpose, we start with the following
\begin{prop}
	\label{PQ}
For the $C^*$-dynamical system $(\ga,\F)$, with $\F$ gapped and $P$ given in \eqref{gaap}, $P\ga$ is a norm closed operator system.
\end{prop}
\begin{proof}
We first note that $P\mathfrak{A}\subset\ga$ is closed in norm since by definition it is the range of $P$, which is a continuous
projection. In addition, $P\ga$ contains the identity $\idd:=\idd_\ga$ of $\ga$ because $P\idd=\idd$.

All is left to do is verify that $P\ga$ is selfadjoint, that is $P\ga=(P\ga)^*$.
This is a straightforward consequence of the fact that both $Q$ and $P$ given in  \eqref{gaap}, are $*$-preserving maps.

Since $Q+P=I$, where $I\equiv \idd_{\mathcal{B}(\ga)}$, it is enough to verify that $Q$ preserves the involution of $\ga$, which is seen as follows.
If $x\in\ga$ and $\g_r$ is the circle of radius $0<r<1$ leaving outside only $\s_{\rm ph}(\F)$, then
\begin{align*}
(Qx)^*=&\bigg(\frac1{2\pi\imath}\rcircleleftint_\g R_\F(\l)x\di\l\bigg)^*
=\bigg(\frac{r}{2\pi}
  \int_{0}^{2\pi} e^{\imath\theta} R_\F(r e^{\imath\theta})x{\rm d}\theta\bigg)^*\\
=&\frac{r}{2\pi}
  \int_{0}^{2\pi}\big(e^{\imath\theta} R_\F(r e^{\imath\theta})x\big)^*{\rm d}\theta
 =\frac{r}{2\pi}
  \int_{0}^{2\pi} e^{-\imath\theta} R_\F(r e^{-\imath\theta})x^*{\rm d}\theta\,,
 \end{align*}
where we have used 
$\big(R_\F(\l)x\big)^*= R_\F(\bar{\l})x^*$, which easily follows from $\Phi$ being a real map.
Now,
\begin{align*}
(Qx)^*=&\frac{r}{2\pi}
  \int_{0}^{2\pi} e^{-\imath\theta} R_\F(r e^{-\imath\theta})x^*{\rm d}\theta
=-\frac{r}{2\pi}
  \int_{0}^{-2\pi} e^{\imath\theta} R_\F(r e^{\imath\theta})x^*{\rm d}\theta\\
=&\frac{r}{2\pi}
  \int_{-2\pi}^{0} e^{\imath\theta} R_\F(r e^{\imath\theta})x^*{\rm d}\theta  
= \frac{r}{2\pi}
  \int_{0}^{2\pi} e^{\imath\theta} R_\F(r e^{\imath\theta})x^*{\rm d}\theta\\
 =&\frac1{2\pi\imath}\rcircleleftint_\g R_\F(\l)x^*\di\l=Qx^*\,,
 \end{align*}
which ends the proof.
\end{proof}
\begin{rem}
Since $P$ commutes with $\Phi$, the operator system $\mathcal{S}:=P\ga$ is clearly invariant
under $\Phi$, that is $\Phi(\mathcal{S})\subset\mathcal{S}$ and thus $\Phi(\mathcal{S})=\mathcal{S}$ because $\Phi\lceil_\mathcal{S}$ is invertible.
\end{rem}

The problem we are going to discuss deals with the possibility to define a new product $\circ$ that
makes $\mathcal{S}$ into a $C^*$-algebra in such a way that the restriction of
$\Phi$ to $\mathcal{S}$ is a $*$-isomorphism with respect to the new product.
We start with an easy yet motivating example when this can be certainly done.

\begin{examp}
Let $\Phi$ be a, not necessarily unital,  $*$-endomorphism of the $C^*$-algebra $\ga$ such that $\s(\F)\subset\{0\}\bigcup\bt$.\footnote{By Proposition \ref{findimendo}, this is certainly true when $\ga$ is finite dimensional.} 
In this case, $\F$ cannot be injective, for otherwise it would be isometric and its spectrum should be the whole
disk. But then $Q$  projects onto $\ker(\F)$, which is a nontrivial two-sided ideal. Therefore, it follows that $P=I-Q$
projects onto ${\rm Ran}(\F)$ because,  for every $x\in\mathfrak{A}$ one has 
$$P(\F(x))=\F(x)-Q(\F(x))=\F(x)-\F(Qx)=\F(x).$$
 But then $P\ga={\rm Ran}(\F)\sim\ga/\ker(\F)$.

In other terms, $\mathcal{S}=P\ga$ is already a $C^*$-algebra (with  $\Phi(\idd_\ga)=\idd_{P\ga}$)
without the product
being changed. Moreover, since the restriction of $\Phi$ to
$\mathcal{S}$ is injective, $\Phi$ acts on $\mathcal{S}$ as a $*$-automorphism.
\end{examp}

\begin{examp}
\label{exop}
The extreme opposite is when the $*$-endomorphism $\F:\ga\to\ga$, supposed to be unital for simplicity, is injective but not surjective.\footnote{Such a situation can occur only if 
$\ga$ is infinite dimensional.} In such a situation, $\s(\F)=\bd$. However, putting $\gb:=\bigcap_{n\in\bn}\F^n(\ga)$, we easily recover that $\F\lceil_\gb$ is a $*$-automorphism possibly trivial in the case $\gb=\bc I$. Since
$\F$ is not gapped, it is expected that $\gb$ does not have a topological complement in $\ga$.
\end{examp}

The following result  ({\it cf.} \cite{K}, Theorem 2.1) deals with the general situation of a completely positive map on a finite dimensional 
$C^*$-algebra.
\begin{prop}
\label{newprod}
Let $\F:\ga\to\ga$ be a completely positive unital map of the finite dimensional $C^*$-algebra $\ga$. Then there exists a subsequence $(n_j)_j\subset\bn$ such that
$P=\lim_j\F^{n_j}$. 

Therefore $P$ is a completely positive projection, and thus $P\ga$ is a $C^*$-algebra when it is equipped with the product
\begin{equation}
\label{effp}
a\circ b:= P(ab)\,,\quad a,b\in P\ga\,.
\end{equation}
\end{prop}
\begin{proof}
Since $\s(\F)$, and therefore $\s_{\rm ph}(\F)$, is finite, there exists a nonzero natural number $m$ such that $\l^m=1$, $\l\in\s_{\rm ph}(\F)$. We then argue that, on the one hand 
$\F^m\lceil_{P\ga}=P$, and on the other hand ({\it cf.} Proposition \ref{ergorisultato}) $\lim_j\F^{mj}\lceil_{Q\ga}=0$.

Therefore, $P$ is a completely positive projection and \eqref{effp} does define a $C^*$-product by mimicking the proof of Theorem 3.1 in \cite{CE}.
\end{proof}
The above result suggests a new way to manage the question of the decoherence, including more general situations at least for gapped completely positive maps $\F$. The first step is indeed to define a new product that changes an operator system to a fully-fledged $C^*$-algebra, which  can always be done in the finite dimensional case, as we proved in Proposition \ref{newprod}. 

We summarise the situation in the following
\begin{rem}
\label{decch}
Let $(\ga,\F)$ be a $C^*$-dynamical system with $\F$ gapped. If
\begin{itemize}
\item[(i)] the projection $P$ given in \eqref{gaap} is completely positive, and thus $\cs_o:=P\ga$ equipped with the map \eqref{effp} which is indeed a $C^*$-product, is a unital $C^*$-algebra,
\item[(ii)] $\F\lceil_{\cs_o}$ is a $*$-automorphism w.r.t. the product \eqref{effp}, 
\end{itemize}
then $\big(\cs_o,\F\lceil_{\cs_o}\big)$ is a Hamiltonian $C^*$-dynamical system.
\end{rem}
At this stage, it is very natural to address the following question.
\begin{conj}
\label{cpconj}
Let $\F:\ga\to\ga$ be a unital completely positive map acting on the unital $C^*$-algebra $\ga$ such that $\s(\F)\subset\bt$. Then $\F$ is a $*$-automorphism.
\end{conj}
\begin{rem}
An immediate consequence of the above conjecture would be the following: for the unital completely positive map $\F:\ga\to\ga$, with $\ga$ finite dimensional and $\cs_o$ equipped with the product \eqref{effp}, $\big(\cs_o,\F\lceil_{\cs_o}\big)$ provides a Hamiltonian $C^*$-dynamical system.
\end{rem}
\begin{proof}
It is enough to note that $\cs_o$ equipped with the product $\circ$ is a $C^*$-algebra and $\F\lceil_{\cs_o}$ is a unital completely positive map with 
$\s\big(\F\lceil_{\cs_o}\big)\subset\bt$.
\end{proof} 
For gapped unital completely positive maps $\F:\ga\to\ga$ for which \eqref{effp} defines a product on $\cs$, we denote by $\cs_o$ the $C^*$-algebra made of $\cs$ endowed with the product $\circ$ in \eqref{effp}.
\begin{rem}
\label{ism}
Let $(\ga,\F)$ be a $C^*$-dynamical system with $\ga$ finite dimensional. If the restriction $\F\lceil_{\cs_o}$
endowed with the $C^*$-product \eqref{effp} provides a
Hamiltonian $C^*$-dynamical system, then $\F$ is isometric on $\cs$ equipped with the original norm. 
\end{rem}
\begin{proof}
The assertion easily follows by the structure of $\cs_o$ given in \cite{CE}, Theorem 7.1. Indeed, by using the notations of such a paper, for the involved matrices we have
$$
X=\begin{pmatrix} 
	 A &0&0\\
	0&A&B\\
	0&C&D\\
     \end{pmatrix}\,,\quad
\F(X)=\begin{pmatrix} 
	 A_\F &0&0\\
	0&A_\F&B_\F\\
	0&C_\F&D_\F\\
     \end{pmatrix}\,.
$$
Taking into account that, for the generic element $X$ as above, the new norm is $\bigg\|\begin{pmatrix}
	A&B\\
	C&D\\
     \end{pmatrix}\bigg\|$,  $\|X\|=\bigg\|\begin{pmatrix}
	A&B\\
	C&D\\
     \end{pmatrix}\bigg\|$, and $\F$ is supposed to be isometric w.r.t. this new norm, we get
$$
\|X\|=\bigg\|\begin{pmatrix}
	A&B\\
	C&D\\
     \end{pmatrix}\bigg\|=\bigg\|\begin{pmatrix}
	A_\F&B_\F\\
	C_\F&D_\F\\
     \end{pmatrix}\bigg\|=\|\F(X)\|\,.
$$
\end{proof}
\begin{examp}
\label{papy}
We would like to revisit Example 2 in \cite{CSU}, where it is explained what can happen for the variables associated to the peripheral spectrum. 

For any fixed unital linearly independent vectors $u,v\in\bc^n$, a linear operator $\mathcal{L}$ is defined on $\bm_n(\bc)$ by
$$
\mathcal{L}(A):=\bigg[\langle Au,u\rangle\langle \,{\bf\cdot}\,,v\rangle-\frac12\big(\langle \,{\bf\cdot}\,,A^*v\rangle+\langle \,{\bf\cdot}\,,v\rangle A\big)\bigg]v\,,
$$
which is the generator of a Markov semigroup. 
It was proven that the peripheral spectrum of $\mathcal{L}$ ({\it i.e.} those eigenvectors with null real part) is made only by $0$, with eigenspace denoted by $V_o$.

The case $n=3$ provides 
an interesting example. Choosing $v:=e_1$ and $u:=e_2$, where $(e_1, e_2, e_3)$ is the canonical basis of
$\bc^3$, one finds that $V_o$ is the operator system made up of all matrices of the form
$$
x=\begin{pmatrix} 
	\om(A) &0\\
	0&A\\
     \end{pmatrix}
     =\begin{pmatrix} 
	 \a &0&0\\
	0&\a&\b\\
	0&\g&\d\\
     \end{pmatrix}\,,\quad\,\, \a,\b,\g,\d\in\bc\,,
$$
where $A\in\bm_2(\bc)$ and $\om(A):=A_{11}$.

However, although $V_o$ fails to be closed under the product, as explained in \cite{CE}, Section 7 or by direct inspection,
it can nonetheless be endowed
with the new product given for $x=\begin{pmatrix} 
	\om(A) &0\\
	0&A\\
     \end{pmatrix}$ and $y=\begin{pmatrix} 
	\om(B) &0\\
	0&B\\
     \end{pmatrix}$,
$$
x\circ y=\begin{pmatrix} 
	\om(AB) &0\\
	0&AB\\
     \end{pmatrix}\,,
$$
such that $(V_o, \circ)\sim\bm_2(\bc)$ as a $C^*$-algebra.

Summarising, in Remark \ref{decch},  (i) is satisfied. We also note that (ii) is trivially satisfied because $V_o={\rm Ker}\mathcal{L}$.
\end{examp}

\section{Examples}\label{examples}

We start by discussing a  simple yet nontrivial class of commutative finite dimensional
examples based on Markov chains, where all  the properties outlined in Remark \ref{decch} hold true.

\begin{examp}

\noindent On $\mathbb{C}^3$ thought of as a commutative $C^*$-algebra $\ga$, where the product is performed component-wise and 
with unity $\idd:=(1, 1, 1)$, we consider the class of positive maps
associated with the stochastic $3\times 3$ matrix $T$ given by
		\begin{equation*}
		T=
		\begin{bmatrix}
		a & b & 1-(a+b)\\
		0 & 0 & 1\\
		0 & 1 & 0
		\end{bmatrix},
		\end{equation*}
where $0\le a<1$ and $0 \le b \le 1$, $a+b \le 1$. The spectrum of $T$ is easily seen to be $\sigma(T)=\{a,-1,1\}$. 
In addition, if we set $\ga_\lambda:={\rm Ker}(\lambda I- T)$, $\lambda\in\mathbb{C}$, after routine
computations one finds:
\begin{itemize}
	\item[{\bf -}] $\mathfrak{A}_{1}=\bc\begin{bmatrix}
	1 \\ 1 \\ 1
	\end{bmatrix}$, $\, \mathfrak{A}_{-1}=\bc\begin{bmatrix}
	\frac{a+2b-1}{1+a} \\ 1 \\ -1
	\end{bmatrix}$,  $\, \mathfrak{A}_a =\bc\begin{bmatrix}1\\0\\0\end{bmatrix}$;
	\item[{\bf -}]  $\idd := \begin{bmatrix}1 \\ 1\\ 1\end{bmatrix}$ , $v := \begin{bmatrix}
	\frac{a+2b-1}{1+a} \\ 1 \\ -1
	\end{bmatrix}$,
	$w := \begin{bmatrix}
	1 \\ 0 \\ 0
	\end{bmatrix}$;
	\item[{\bf -}] $v^2=\idd+\left(\Big(\frac{a+2b-1}{1+a}\Big)^2-1\right)w$;
	\item[{\bf -}] The space $\ga_1 \bigoplus\ga_{ -1}={\rm span}\left\{\begin{bmatrix}
	1 \\ 1 \\ 1
	\end{bmatrix}\, , \,\begin{bmatrix}
	\frac{a+2b-1}{1+a} \\ 1 \\ -1
	\end{bmatrix}\right\}$ is not closed with respect to the algebra product on $\mathbb{C}^3$. For one has
\begin{align*}
&(t_1\idd +t_2v)(t_3\idd + t_4v) = t_1t_3\idd + (t_1t_4 +t_2t_3)v +t_2t_4v^2\\
=&(t_1t_3+ t_2t_4)\idd + (t_1t_4 +t_2t_3)v +t_2t_4\left(\Big(\frac{a+2b-1}{1+a}\Big)^2-1\right)w\\
 \notin&\,\mathfrak{A}_1\bigoplus\ga_ {-1}\,.
\end{align*}
\end{itemize}

Even so, the direct sum $\ga_{1} \bigoplus \ga_{-1}$ is a $C^*$-algebra w.r.t the new product $\circ$.
Let $P: \bc^3\rightarrow\bc^3$ be the projection given by 
$$P(t_1\idd+t_2v+t_3w):=t_1\idd+t_2v,\quad t_1, t_2, t_3\in\bc\,.$$

The restriction of $T$ to $P(\mathbb{C}^3)=\mathfrak{A}_1\bigoplus\mathfrak{A}_{-1}$ is actually a $*$-isomorphism with respect to $\circ$, that is $T(x\circ y)=(Tx)\circ(Ty)$ for any $x, y\in P(\mathbb{C}^3)$.
Indeed, if $x=x_1\idd+x_2v$ and $y=y_1\idd+y_2v$, we have
$x\circ y=P(xy)=(x_1y_1+x_2y_2)\idd+(x_1y_2+x_2y_1)v$, hence $T(x\circ y)=(x_1y_1+x_2y_2)\idd-(x_1y_2+x_2y_1)v$ as $Tv=-v$.
On the other hand, $Tx\circ Ty= (x_1\idd-x_2v)\circ(y_1\idd-y_2v)=(x_1y_1+x_2y_2)-(x_1y_2+x_2y_1)$, and the claimed equality is thus seen to hold.

\medskip

A second class of working examples is provided by matrices of the form
\begin{equation*}
T = \begin{bmatrix}
0& 0& 1\\
0 & 1-a & a\\
1 & 0 & 0
\end{bmatrix} 
\end{equation*}
where $a$ is a parameter with $0<a<1$. The spectrum of $T$ is given by $\sigma(T) = \{1, -1, 1-a\}$.
The corresponding eigenspaces are now
\begin{equation*}
\mathfrak{A}_1 = \mathbb{C}\begin{bmatrix}
1 \\ 1 \\ 1\end{bmatrix}, \quad \mathfrak{A}_{-1} = \mathbb{C}\begin{bmatrix}
1 \\ \frac{a}{2-a} \\ -1\end{bmatrix}, \quad \mathfrak{A}_{1-a} = \mathbb{C}\begin{bmatrix}
0 \\ 1 \\ 0
\end{bmatrix}.
\end{equation*}
As before, let us set

\begin{equation*}
\idd := \begin{bmatrix}
1 \\ 1 \\ 1
\end{bmatrix}, \quad v:= \begin{bmatrix}
1 \\ \frac{a}{2-a} \\ -1
\end{bmatrix}, \quad w:= \begin{bmatrix}
0 \\ 1 \\ 0
\end{bmatrix},
\end{equation*}
Now $(t_1 \idd + t_2 v)(t_3 \idd + t_4 v) = (t_1t_3+t_2t_4)\idd + (t_1t_4+t_2t_3)v +t_2t_4[a^2/(2-a)^2-1]w=
\notin \mathfrak{A}_{1}\bigoplus\ga_{-1}$.

Once again, in the direct sum $\ga_{1} \bigoplus \ga_{-1}$ it is still  possible to define the new product $\circ$, and verify that the map $T$ restricts to $\ga_1\bigoplus\ga_{-1}$ as a $*$-isomorphism w.r.t. the above product. We leave to the reader all remaining calculations.
\end{examp}

At this stage, it is natural to address the question, whose answer is probably affirmative, whether any Markov chain provides a Hamiltonian dynamical system when restricted to its persistent portion, after changing there the product. 

\medskip 

We next move on to provide working non commutative finite dimensional examples as well.

\begin{examp}

Now in light of a well-known result in \cite{Choi}, the most general completely positive map on a matrix algebra
$M_n(\mathbb{C})$ assumes the form
$$
\Phi(A)=\sum_{i=1}^kV_i^*AV_i\,,\quad A\in M_n(\mathbb{C})\,,
$$
where $k$ is any positive natural number, and $V_i$ an element of $M_n(\mathbb{C})$ for any $i=1, 2, \ldots, n$. Note that
$\Phi$ is unital if and only if $\sum_{i=1}^k V_i^*V_i=I$.\\

For such a purpose, we consider an elementary example for $n=k=2$ by looking at the completely positive map $\Phi$ that corresponds to taking $V_1
=\left(\begin{array}{cc}
0 & 1\\
0& 0\\
\end{array}
\right)
$
and $V_2=V_1^*=\left(\begin{array}{cc}
0 & 0\\
1& 0\\
\end{array}
\right)$\,.
Accordingly, one finds
$$\Phi\left(\left(\begin{array}{cc}
a & b\\
c& d\\
\end{array}
\right)\right)=\left(\begin{array}{cc}
d & 0\\
0& a\\
\end{array}
\right)$$
Eigenvalues and eigenspaces of $\Phi$ are found at once. 

The spectrum $\sigma(\Phi)$ is the set
$\{0, 1, -1\}$, and the corresponding eigenspaces are $\ga_o:=\ker(\Phi)={\rm span}\{V_1, V_2\}$,
$\ga_1:=\ker(I-\Phi)=\mathbb{C}I$, and
$\ga_{-1}:=\ker(I+\Phi)=\mathbb{C}A$, where $A$ is the matrix $\left(\begin{array}{cc}
1 & 0\\
0& -1\\
\end{array}\right)$.

Now the direct sum $\ga_1\bigoplus\ga_{-1}$ is already a $*$-subalgebra since $A^2=1$. Moreover,
the restriction of $\Phi$ to  $\ga_1\bigoplus\ga_{-1}$ preserves the product. Indeed, for any
$t, t',s,s'\in\mathbb{C}$ we have:
\begin{eqnarray*}
&&\Phi\left((tI+sA)(t'I+s'A)\right)= \Phi(tt'I +ts'A +t'sA+ss'I)=\\
&&(tt'+ss')I-(t's+ts')A= (t1-sA)(t'I-s'A)=\\
&& \Phi(tI+sA)\Phi(t'I+s'A)\,.
\end{eqnarray*}

\end{examp}
 
\medskip
 
Notice that, for such a simple example,  Conjecture \ref{cpconj} holds true.  In addition, one might also object that it is too elementary in that  $\ga_1\bigoplus\ga_{-1}$ is already
a $*$-subalgebra and the restriction of
$\Phi$ to it already preserves the product. However, it is not too difficult to conceive slightly more sophisticated  examples
where the product does need to be changed. 

\begin{examp}

By using the conditional expectation
$E$ from $\bm_n(\bc)$ onto the diagonal subalgebra $\bc^n$, that is the subalgebra of all diagonal matrices, given by
$$
E(A)_{ij}=\delta_{i,j}A_{ij}, \quad A=(A_{ij})_{i, j=1}^n\in \bm_n(\bc)\,,
$$ 
any positive map
$T$ on $\bc^n$ can be extended to a completely positive map $\Phi$ on $\bm_n(\bc)$ as $\Phi:=T\circ E$. 
In addition, the spectrum of
$\F$ is the same as the spectrum of $T$ up to $0$, as shown below.
\begin{lem}
With the notations set above, $\sigma(\F)=\sigma(T)\bigcup\{0\}$. 
\end{lem}
\begin{proof}
Since $\F$ is not injective, $0$ belongs to $\sigma(\F)$, and in addition
the inclusion $\sigma(T)\cup\{0\}\subset\sigma(\F)$ is obvious because the restriction of $\F$ to $\bc^n$ is $T$. Therefore,
all we have to do is prove that if $\lambda\neq 0$ sits in $\sigma(\F)$, then $\lambda$ also sits in
$\sigma(T)$. Let $A\in \bm_n(\bc)$ such that 
$\F(A)=\lambda A$, that is
$T(E(A))=\lambda A$. From $A=\frac{1}{\lambda}T(E(A))$, we see that $A$ must be in $\bc^n$. But then,
$A=E(A)$, which means $\lambda\in\sigma(T)$ because 
$$
T(A)=T(E(A))=\F(A)=\lambda A\,.
$$
\end{proof}

As an immediate result of the above lemma, whenever one is given a positive map $T$ on $\bc^n$ with the properties in Remark \ref{decch}, then its extension
$\F=T\circ E$ to $\bm_n(\bc)$ still have the same properties.
In particular, this applies to the commutative examples we discussed at the beginning of the present section, thus providing
instances of completely positive maps on a full matrix algebra such that the conditions in
Remark \ref{decch} are fulfilled in a nontrivial way.

\end{examp}

\section*{Acknowledgements}

We would like to thank Jochen Gl\"{u}ck for many useful discussions.
We also owe a debt of gratitude to the anonymous referee for his or her careful reading of
the manuscript.
The first-named author acknowledges the \lq\lq MIUR Excellence Department Project'' awarded to the Department of Mathematics, University of Rome Tor Vergata, CUP E83C18000100006.

\end{document}